\documentclass[11pt, reqno]{amsart}
\usepackage[margin=1in]{geometry}
\usepackage{amsmath, amssymb, amsthm,bbm,bm}
\usepackage[shortlabels]{enumitem}
\usepackage{url}
\usepackage[hidelinks]{hyperref}
\usepackage[capitalize]{cleveref}
\usepackage[noadjust]{cite}

\theoremstyle{plain}
\newtheorem{thm}{Theorem}
\newtheorem{lemma}[thm]{Lemma}
\newtheorem{cor}[thm]{Corollary} 
\newtheorem{prop}[thm]{Proposition}

\newtheorem{conj}[thm]{Conjecture}

\theoremstyle{definition}

\theoremstyle{remark}

\numberwithin{thm}{section}
\numberwithin{equation}{section}

\Crefname{thm}{Theorem}{Theorems}

\newcommand{\abs}[1]{\left\lvert#1\right\rvert}

\newcommand{\ceil}[1]{\left\lceil #1 \right\rceil}

\newcommand{\E}{\mathbb E}

\newcommand{\cA}{\mathcal{A}}
\newcommand{\cB}{\mathcal{B}}

\DeclareMathOperator{\ex}{ex}

\title{Triangle Ramsey numbers of complete graphs}
\author[Fox]{Jacob Fox}
\author[Tidor]{Jonathan Tidor}
\author[Zhang]{Shengtong Zhang}
\address{Department of Mathematics, Stanford University, Stanford, CA 94305, USA}
\email{\textnormal{\{}jacobfox, stzh1555\textnormal{\}}@stanford.edu}
\address{Department of Mathematics, Princeton University, Princeton, NJ 08544, USA}
\email{jtidor@princeton.edu}

\thanks{Fox was supported by NSF Awards DMS-2452737 and DMS-2154129. Tidor was supported by a Stanford Science Fellowship.}
\date{}

\begin{document}

\begin{abstract}
A graph is $H$-Ramsey if every two-coloring of its edges contains a monochromatic copy of $H$. Define the $F$-Ramsey number of $H$, denoted by $r_F(H)$, to be the minimum number of copies of $F$ in a graph which is $H$-Ramsey. This generalizes the Ramsey number and size Ramsey number of a graph. Addressing a question of Spiro, we prove that
\[r_{K_3}(K_t)=\binom{r(K_t)}3\]
for all sufficiently large $t$.  We do so through a result on graph coloring: there exists an absolute constant $K$ such that every $r$-chromatic graph where every edge is contained in at least $K$ triangles must contain at least $\binom r3$ triangles in total. 
\end{abstract}

\maketitle

\section{Introduction}

A graph $G$ is {\it $H$-Ramsey} if every two-coloring of the edges of $G$ contains a monochromatic copy of $H$. We define $r_F(H)$ to be the minimum number of copies of $F$ in a graph which is $H$-Ramsey. The usual notion of the Ramsey number corresponds to the case when $F$ is a single vertex, while the size Ramsey number corresponds to the case when $F = K_2$.

In general, the relationship between the Ramsey number and the size Ramsey number is not known. However, when $H$ is a complete graph the following observation of Chv\'atal (as reported in the paper of Erd\H{o}s, Faudree, Rousseau, and Schelp \cite{EFRS78} that defines the size Ramsey number) gives the precise relationship, showing that the complete graph on $r(K_t)$ vertices has the minimum number of edges among all graphs which are $K_t$-Ramsey. 

\begin{thm}[{Chv\'atal \cite[Theorem 1]{EFRS78}}]
\label{thm:chvatal-size-ramsey-clique}
For all $t$,
\[r_{K_2}(K_t)=\binom{r(K_t)}{2}.\]
\end{thm}

In this paper we study a question of Sam Spiro who asked if an analogue of Chv\'atal's result holds for $r_{K_s}(K_t)$. We conjecture that the answer is yes.

\begin{conj}
\label{conj:clique-clique}
For all $s\leq t$,
\[r_{K_s}(K_t)=\binom{r(K_t)}s.\]
\end{conj}

We verify this conjecture in the case $s = 3$ and $t$ sufficiently large. We call $r_{K_3}(K_t)$ the \emph{triangle Ramsey number}. 

\begin{thm}
\label{thm:main-triangle-ramsey}
For all sufficiently large $t$,
\[r_{K_3}(K_t)=\binom{r(K_t)}3.\]
\end{thm}

To the best of our knowledge, this is the first paper to study $r_F(H)$ for any $F$ other than a vertex or an edge. We think that studying $r_F(H)$ for various pairs $(F,H)$ is an interesting avenue of research. The way in which this generalizes classical Ramsey numbers is similar to the generalized Tur\'an function $\ex(n,F,H)$ recently introduced by Alon and Shikhelman \cite{AS16}. This function is defined to be the maximum number of copies of $F$ in an $H$-free graph and has already received much study (see, e.g., \cite{Zyk49,Erd62,Gre12,HHKNR13,GP19,GS20,Luo18,MQ20}).

Observe that the size Ramsey number of a $k$-uniform hypergraph is a special case (when $F$ is a single edge, the complete $k$-uniform hypergraph on $k$ vertices) of the straightforward generalization of $F$-Ramsey numbers to $k$-uniform hypergraphs. Note that the triangle Ramsey number is superficially similar to the size Ramsey number of 3-uniform hypergraphs. However, the natural extension of Conjecture \ref{conj:clique-clique} to hypergraphs does not hold.  
 Indeed, McKay \cite{McK17} showed that the size Ramsey number satisfies $\hat{r}(K_t^{(3)})<\binom{r(K_t^{(3)})}3$ for $t=4$. The contrast with \cref{thm:main-triangle-ramsey} and the growth rates of these functions suggest that the two problems are not deeply related. See \cite{DLMR17} for more discussion of the size Ramsey number of hypergraphs.

Our proof of \cref{thm:main-triangle-ramsey} follows from a result on the number of triangles in graphs with chromatic number $r$. First, we fix some terminology. 

We say a graph $G$ is \emph{$r$-chromatic} if $\chi(G)\geq r$. A graph $G$ is called \emph{$r$-critical} if $G$ is $r$-chromatic, but no proper induced subgraph of $G$ is $r$-chromatic. We say a graph $G$ is \emph{$H$-Ramsey critical} if $G$ is $H$-Ramsey but no proper subgraph of $G$ is $H$-Ramsey. The following elementary proposition is the basis of our argument. We will give a proof of it in \cref{sec:basic-turan}.

\begin{prop}
\label{prop:ramsey-critical}
\begin{enumerate}[(a)]
\item If $G$ is $K_t$-Ramsey,  then $G$ is $r(K_t)$-chromatic. 
\item If $G$ is $K_t$-Ramsey critical, then every edge of $G$ is contained in a $K_t$. 
\end{enumerate}
\end{prop}

The folklore proof of \cref{thm:chvatal-size-ramsey-clique} follows from \cref{prop:ramsey-critical}(a) and the simple observation that every $r$-chromatic graph has at least $\binom r2$ edges. If it were true that every $r$-chromatic graph has at least $\binom r3$ triangles, we could deduce our result on the triangle Ramsey number. Unfortunately this is not true; there are triangle-free $r$-chromatic graphs. However, any such graph has $\Omega(r^3\log^{2} r)$ edges \cite{Nil00}. By \cref{prop:ramsey-critical}(b), any graph with so many edges which is $K_t$-Ramsey critical will have significantly more than $\binom r3$ triangles. In light of \cref{prop:ramsey-critical}, it suffices to study $r$-chromatic graphs where every edge is contained in a $K_t$. We find it simpler to relax this condition to only requiring that each edge is contained in many triangles.

\begin{thm}
\label{thm:main-chromatic-weak}
There exists an absolute constant $K$ such that the following holds. Let $G$ be an $r$-chromatic graph where every edge is contained in at least $K$ triangles. Then $G$ contains at least $\binom r3$ triangles.
\end{thm}

 Together \cref{prop:ramsey-critical,thm:main-chromatic-weak} imply \cref{thm:main-triangle-ramsey} for $t\geq K+2$. One could also consider $q$-color generalization of all of these problems. The $q$-color generalizations of \cref{thm:chvatal-size-ramsey-clique,thm:main-triangle-ramsey,prop:ramsey-critical} hold by the same proofs. We also conjecture that the $q$-color generalization of  \cref{conj:clique-clique} holds. 

The remainder of the paper is devoted to the proof of \cref{thm:main-chromatic-weak}. In \cref{sec:basic-turan}, we introduce a short ``Tur\'{a}n argument" that proves \cref{thm:main-triangle-ramsey} asymptotically. To refine this argument and achieve the exact bound, we draw from a wide range of works on graph coloring in locally sparse graphs, which we summarize in \cref{sec:coloring}. One of the main lemmas in that section, \cref{thm:degen-coloring}, uses a probabilistic coloring argument and relies on concentration inequalities, including the Azuma-Hoeffding inequality and Talagrand's inequality. We first prove \cref{thm:main-chromatic-weak} for graphs with $O(r)$ vertices in \cref{sec:linear-size} and then build upon this result to prove the full theorem in \cref{sec:main-proof}. Finally, we discuss some further directions of research in \cref{sec:conclusion}.

\medskip
\noindent \emph{Notation:} We mostly use standard graph theory notation. For a graph $G$ we write $|G|$ for the number of vertices, $e(G)$ for the number of edges, and $t(G)$ for the number of triangles.

\medskip
\noindent \emph{Acknowledgements:} We learned the problem of determining $r_F(H)$ and \cref{conj:clique-clique} from Sam Spiro at the 2023 Workshop on Ramsey Theory held at UCSD. We thank Sam Spiro for telling us of the problem, David Conlon for helpful conversations,  and Dhruv Mubayi, Andrew Suk, and Jacques Verstraete for organizing the workshop.

\section{The basic Tur\'an argument}
\label{sec:basic-turan}

In this section, we give a short proof of the asymptotic bound $r_{K_3}(K_t)\geq(1-o(1))\binom{r(K_t)}3$. This follows from a clever application of Tur\'an's theorem. We give this proof not only for its aesthetic appeal, but also because we build upon this result and its proof technique later in the paper.

First, we provide a proof of \cref{prop:ramsey-critical} for completeness. 

\begin{proof}[Proof of \cref{prop:ramsey-critical}]
Write $r=r(K_t)$. By definition, there is a two-edge-coloring of $K_{r-1}$ which avoids monochromatic copies of $K_t$. Call this coloring $\phi\colon\binom{[r-1]}2\to[2]$.

Suppose for the sake of contradiction that $G$ is $(r-1)$-colorable. Thus, there is a partition $V(G)=V_1\sqcup\cdots\sqcup V_{r-1}$ into independent sets. We can construct a two-edge-coloring of $G$ by assigning $e=\{u,v\}$ between the vertex parts $V_i,V_j$ the color $\phi(\{i,j\})$. Then any monochromatic copy of $K_t$ in this coloring of $G$ has vertices in distinct parts, so it corresponds to a monochromatic copy of $K_t$ in the coloring of $K_{r-1}$. This is a contradiction, proving part (a).

If $G$ is $K_t$-Ramsey, clearly removing all edges of $G$ that are not contained in a $K_t$ preserves this property. This proves part (b).
\end{proof}

It is worth pointing out that \cref{prop:ramsey-critical}(b) can be strengthened to the following. If $G$ is $K_t$-Ramsey critical and $f$ is an edge of $G$, then every edge-coloring of $G\setminus f$ with colors red and blue without a monochromatic $K_t$ has the property that there is a red copy of $K_t\setminus e$ in which $f$ is the missing edge, and a blue copy of $K_t\setminus e$ in which $f$ is the missing edge. Indeed, if not, we could extend this coloring to a red-blue edge-coloring of $G$ without a monochromatic $K_t$, contradicting $G$ is $K_t$-Ramsey critical. In particular, this implies that any edge of a $K_t$-Ramsey critical graph is the intersection of two copies of $K_t$ in the graph. 

We now begin the Tur\'an argument. We repeatedly use the following version of Tur\'an's theorem, which follows from the standard version applied to the complement of the graph.

\begin{thm}[Tur\'an]
\label{thm:turan}
For any graph $G$,
\[e(G)\geq\frac{|G|^2}{2\alpha(G)}-\frac{|G|}2.\]
\end{thm}

The next theorem is the main result of the section. 

\begin{thm}
\label{thm:basic-turan-bound}
For any $r$-chromatic graph $G$,
\[t(G)+\frac12 e(G)\geq\binom r3+\frac12 \binom r2.\]
\end{thm}

\begin{proof}
We construct a sequence of (induced) subgraphs of $G$
\[G\supseteq G_0\supseteq G_0'\supseteq G_1\supseteq G_1'\supseteq \cdots \supseteq G_{r-1}\supseteq G_{r-1}'\] satisfying $\chi(G_i)=\chi(G_i')=r-i$.

Let $G_0$ be any subgraph of $G$ satisfying $\chi(G_0)=r$. Given $G_i$ (which satisfies $\chi(G_i)=r-i$), let $G_i'$ be an $(r-i)$-critical subgraph of $G_i$. Clearly $\chi(G_i')=r-i$. Then let $I_i$ be an independent set in $G_i'$ with maximal size and define $G_{i+1}=G_i'\setminus I_i$. It follows that $\chi(G_{i+1})=r-i-1$. To see this, note that $\chi(G_{i+1})<\chi(G_i')$ since $G_i'$ is critical, and $\chi(G_i')\leq 1+\chi(G_{i+1})$ since $I_i$ is an independent set. 

For a graph $H$ and a vertex $v$, write $t_H(v)$ for the number of triangles in $H$ that contain $v$. By Tur\'an's theorem (\cref{thm:turan}) we have the bound
\[t_H(v)=e(G[N_H(v)])\geq\frac{d_H(v)^2}{2\alpha(H)}-\frac{d_H(v)}{2}.\]

Since $G_i'$ is $(r-i)$-critical, we know that $\delta(G_i')\geq r-i-1$. Summing the above inequality gives
\[\sum_{i=0}^{r-1}\sum_{v\in I_i}t_{G_i'}(v)\geq\sum_{i=0}^{r-1}\sum_{v\in I_i}\frac{d_{G_i'}(v)^2}{2\alpha(G_i')}-\frac{d_{G_i'}(v)}{2}\geq\sum_{i=0}^{r-1}\frac{(r-i-1)^2}2-\frac 12\sum_{i=0}^{r-1}\sum_{v\in I_i}d_{G_i'}(v).\]
Since $I_i$ is an independent set, it follows that $\sum_{v\in I_i}t_{G_i'}(v)$ is the number of triangles removed when producing $G_{i+1}$ from $G_i'$. Thus $\sum_{i=0}^{r-1}\sum_{v\in I_i}t_{G_i'}(v)\leq t(G)$. Similarly, $\sum_{i=0}^{r-1}\sum_{v\in I_i}d_{G_i'}(v)\leq e(G)$. The first term on the right-hand side evaluates to $\binom r3+\tfrac12\binom r2$, implying the desired result.
\end{proof}

\begin{cor}
\label{cor:ramsey-bound-weak}
$r_{K_3}(K_t)\geq(1-o(1))\binom{r(K_t)}3$.
\end{cor}
\begin{proof}
Write $r=r(K_t)$. Let $G$ be a $K_t$-Ramsey critical graph. By \cref{prop:ramsey-critical}, we know $G$ is $r$-chromatic and every edge is contained in a $K_t$, and thus is contained in at least $t-2$ triangles. This implies that $(t-2)e(G)\leq 3t(G)$. Thus by \cref{thm:basic-turan-bound}
\[\binom r3+\frac12\binom r2\leq t(G)+\frac 12 e(G)\leq \left(1+\frac3{2(t-2)}\right)t(G).\qedhere\]
\end{proof}

\section{Coloring results}
\label{sec:coloring}

Here we collect a handful of coloring results, starting with Gallai's theorem on small critical graphs. This allows us to prove \cref{thm:main-chromatic-weak} for graphs with at most $2r-2$ vertices. We then recall some results bounding the chromatic number of locally sparse graphs and prove a new result of this sort.

The join of two graphs is formed by taking their disjoint union and then adding all edges from one graph to the other.

\begin{thm}[{Gallai \cite{Gal63}}]
\label{thm:gallai}
For any $r$-critical graph $G$, if $|G|\leq 2r-2$, then $G$ can be written as the join of two smaller critical graphs.
\end{thm}

\begin{cor}
\label{cor:gallai-triangle}
For any $r$-chromatic graph $G$, if $|G|\leq 2r-2$, then $t(G)\geq\binom r3$.
\end{cor}

\begin{proof}
Let $G'$ be an $r$-critical subgraph of $G$, satisfying $|G'|\leq|G|\leq 2r-2$. By Gallai's theorem (\cref{thm:gallai}), we can write $G'$ as the join of an $a$-critical graph $A$ and a $b$-critical graph $B$ for some $a,b\geq 1$ satisfying $a+b=r$.

We have the bounds $|A|\geq a$ and $e(A)\geq \binom a2$. By \cref{thm:basic-turan-bound} we also have $t(A)+\tfrac12 e(A)\geq \binom a3+\tfrac12\binom a2$. Similarly the analogous bounds hold for $B$. Then
\begin{align*}
t(G')
&=t(A)+e(A)|B|+|A|e(B)+t(B)\\
&=\left(t(A)+\frac12e(A)\right)+e(A)\left(|B|-\frac12\right)+e(B)\left(|A|-\frac12\right)+\left(t(B)+\frac12e(B)\right)\\
&\geq \binom a3+\frac12\binom a2+\binom a2\left(b-\frac12\right)+\binom b2\left(a-\frac12\right)+\binom b3+\frac12\binom b2\\
&= \binom {a+b}3=\binom{r}3.\qedhere
\end{align*}
\end{proof}

A powerful result of Alon--Krivelevich--Sudakov \cite{AKS99} bounds the chromatic number of locally sparse graphs. We use an extension due to Vu \cite{Vu02} on the list chromatic number.

\begin{thm}[{Vu \cite[Theorem 1.4]{Vu02}}]
\label{thm:loc-sparse-list-coloring}
There is a constant $K$ such that the following holds. Let $G$ be a graph and let $1<f<\Delta(G)^2$ be chosen such that $e(G[N(v)])\leq\Delta(G)^2/f$ for all $v\in G$. Then
\[\chi_\ell(G)\leq K\Delta(G)/\log f.\]
\end{thm}

The following result of Harris gives an upper bound on the chromatic number of an $n$-vertex graph that depends on the number of triangles in the graph and the maximum number of triangles a vertex is in. 

\begin{thm}[{Harris \cite[Theorem 2.3]{Har19}}]
\label{thm:loc-sparse-vertex-triangle-coloring}
There is a constant $K$ such that the following holds. Let $G$ be a graph with $n$ vertices and $t$ triangles, such that every vertex is contained in at most $y$ triangles. We have
\[\chi(G)\leq K\left(\sqrt{\frac{n}{\log n}}+\frac{t^{1/3}{\lg}{\lg}(t^2/y^3)}{{{\lg}}^{2/3}(t^2/y^3)}\right).\]
where ${\lg}(x) = \max(\log(x), 1)$.
\end{thm}

A graph is $d$-degenerate if there is an ordering $v_1, v_2, \ldots, v_n$ of its vertices (called a degeneracy ordering) such that for each $i=1,\ldots,n$, vertex $v_i$ has at most $d$ neighbors $v_j$ with $j<i$.

We also will need a result that shows that certain $d$-degenerate locally sparse graphs have chromatic number slightly smaller than $d$. Many ideas of this proof are inspired by a result of Molloy--Reed \cite{MR02} on the chromatic number of locally sparse graphs with bounded maximum degree. 

For a $d$-degenerate graph with degeneracy order $v_1, v_2, \ldots, v_n$, we write $N^+(v_i)=N(v_i)\cap\{v_1,\ldots,v_{i-1}\}$ for the forward neighborhood of $v_i$. By definition, $|N^+(v_i)|\leq d$ for all $i\in[n]$. 

\begin{lemma}[{cf. \cite[Theorem 10.5]{MR02}}]
\label{thm:degen-coloring}
For all $0<\epsilon, \delta < 1$, there exists $\gamma > 0$ such that the following holds for sufficiently large $d$. Let $G$ be a $d$-degenerate graph with degeneracy order $v_1, v_2, \ldots, v_n$ such that the following two conditions hold. 
\begin{enumerate}
    \item $\abs{N^+(v_j)} \leq d$ for all $1 \leq j \leq n$, and $\abs{N^+(v_j)} \leq (1 - \epsilon) d$ for $1 \leq j \leq n - \delta d / 4$.
    \item For any $v \in G$, we have $e(N^+(v)) \leq (1 - \delta) \binom{d}{2}$.
\end{enumerate}
Then $\chi(G) \leq (1 - \gamma) d$.
\end{lemma}

\begin{proof}
We take $\gamma = \tfrac{\delta}{200}e^{-10 \epsilon^{-1}}$. Then define $d_1 = \ceil{\delta d / 4}$ and $s = \ceil{\epsilon d / 2}$ and $d_2 = \ceil{\gamma d}$. Assume that $d$ is sufficiently large in terms of $\epsilon,\delta,\gamma$.

Let $\cA$ be the vertices $v_1, \ldots, v_{n - d_1}$, and let $\cB$ be the vertices $v_{n - d_1 + 1}, \ldots, v_n$. The idea is to first color a few vertices of $\cA$ in such a way that all vertices in $\cB$ contain many pairs of neighbors that receive the same color. If we can do this, then a greedy coloring in the degeneracy order gives a proper $(d-d_2)$-coloring of $G$.

More precisely, we make the following claim. It is possible to find a partial proper coloring of $\cA$ with $s$ colors, such that for every $v \in \cB$, either $\abs{N^+(v)} \leq d - d_2 - 1$, or at least $d_2+1$ colors appear at least twice in $N^+(v)$.

We produce the desired partial coloring with the following randomized strategy. Color the vertices of $\cA$ with $s$ colors uniformly at random. For each pair of adjacent vertices that receive the same color, uncolor the one that appears later in the degeneracy order.

From now on fix $v\in\cB$ and assume that $|N^+(v)|>d-d_2-1$.

Let $X_v$ be the number of colors that are assigned to at least one pair of non-adjacent neighbors of $v$ and are retained by all neighbors of $v$ that are assigned that color. Let $X'_v$ be the number of colors that are assigned to exactly two neighbors of $v$ and that these two neighbors are non-adjacent and both retain this color. Clearly $X_v\geq X'_v$.

The number of pairs of non-adjacent vertices in $N^+(v)\cap\cA$ is at least
\[\binom{|N^+(v) \cap \cA|}{2}-e(N^+(v)) \geq \binom{|N^+(v)|}{2}-e(N^+(v))-d|\cB|\geq \binom{d-d_2}{2}-(1-\delta)\binom{d}{2}-dd_1\geq\frac{\delta}{5}\binom{d}{2}.\]
Let $u,w$ be one of these pairs of non-adjacent vertices in $N^+(v)\cap\cA$. For $u,w$ to be the only two vertices assigned a color $i$ in $N^+(v)$ and both to retain it, no vertex in $N^+(v) \setminus \{u,w\}$, $N^+(u)$, or $N^+(w)$ can be assigned color $i$. We can then bound \[\E[X_v']\geq s\cdot\frac{\delta}{5}\binom{d}{2}\cdot s^{-2}\cdot(1-s^{-1})^{3d}\geq 3d_2.\]

We write $X_v=AT_v-Del_v$ where $AT_v$ (for ``assigned twice'') is defined to be the number of colors that are assigned to at least one pair of non-adjacent neighbors of $v$ and where $Del_v$ is the number of colors that are assigned to at least one pair of non-adjacent neighbors of $v$ and then deleted from at least one of them. We will show that $AT_v$ and $Del_v$ both concentrate well around their means.

First note that $AT_v$ is a 2-Lipschitz function of the colors assigned to $N^+(v)$. Therefore by the Azuma--Hoeffding inequality
\[\Pr(|AT_v-\E[AT_v]|>t)<2e^{-t^2/8d}.\] Next note that $Del_v$ is a 2-Lipschitz function now of the colors assigned to the second neighborhood of $v$. However $Del_v$ is 3-certifiable in the sense that if $Del_v\geq r$ there are a set of $3r$ vertices whose colors certify the fact that $Del_v\geq r$. Thus by Talagrand's inequality (see, e.g., \cite[Page 81]{MR02})
\[\Pr(|Del_v-\E[Del_v]|>t)\leq 4e^{-(t-120\sqrt{3\E[Del_v]})^2/(96\E[Del_v])}\leq 4e^{-(t-120\sqrt{3d})^2/96d}.\]

Using the facts that $X_v\geq X'_v$ and $X_v=AT_v-Del_v$, we conclude that 
\[\Pr(X_v<d_2)\leq 2e^{-d_2^2/8d}+4e^{-(d_2-120\sqrt{3d})^2/96d}\leq 6e^{-\gamma^2d/100}\]
where the second inequality holds for $d$ sufficiently large in terms of $\gamma$.

Taking a union bound over all $d_1$ vertices $v\in\cB$, we see that for $d$ sufficiently large, with positive probability $X_v\geq d_2+1$ for all $v\in\cB$ with $|N^+(v)|>d-d_2-1$. 

We now claim that a greedy coloring of the uncolored vertices in the degeneracy ordering gives a $(d-d_2)$-coloring of $G$. Since $G[\cA]$ is $(1 - \epsilon) d$-degenerate, when we color $v\in\cA$, at most $(1-\epsilon)d$ colors will have been used on the forward neighborhood. Since the initial partial coloring uses $s$ colors, at most $s$ colors will be used on the backward neighborhood. Thus there will be at least $(d-d_2)-(1-\epsilon)d-s\geq 1$ colors available for $v$. When we color $v\in\cB$, the forward neighborhood of $v$ has size at most $d-d_2-1$ or has at least $d_2+1$ colors appear at least twice, and hence there is at least one color available for $v$. Therefore the greedy procedure produces a proper $(d-d_2)$-coloring of $G$, as desired.
\end{proof}

\section{Linear-sized graphs}
\label{sec:linear-size}

The goal of this section is to prove \cref{thm:main-triangle-ramsey} for graphs with $O(r)$ vertices. We start with the special case where the graph has a clique of almost maximal size.

\begin{lemma}
\label{prop:large-clique}
For every $C$ there exists $\beta>0$ such that the following holds. Let $G$ be an $r$-critical graph with $3r/2\leq |G|\leq Cr$ and $\omega(G)\geq(1-\beta)r$. Then $t(G)>\binom r3$.
\end{lemma}

\begin{proof}
Define $\beta=e^{-8CK}/384$ where $K$ is the constant in \cref{thm:loc-sparse-list-coloring}. Suppose for contradiction that $t(G)\leq \binom r3$.

Let $Q$ be a clique in $G$ with maximal size. Let $N$ be the set of vertices $v\in V(G)\setminus Q$ with $|N(v)\cap Q|\geq r/2$. Write $G'=G\setminus(Q\cup N)$, so $G'$ is the induced subgraph of $G$ by deleting the vertices that are in $Q$ or $N$. Let $D$ be the set of vertices $v\in V(G')$ such that $e(G'[N(v)])\geq 6\beta r^2$. Finally let $H=G'\setminus D$. Note that we can bound the number of triangles in $G$ as

\[t(G)\geq \binom{|Q|}3+|N|\binom{r/2}2+|D|\frac{6\beta r^2}{3}.\]
Since we assumed that $t(G)\leq \binom r3$ and $|Q|\geq(1-\beta)r$, this implies that $|N|\leq 5\beta r<r/8$ and $|D|<r/4$. 

Using $|G|\geq 3r/2$ and $|Q|\leq r$ we can see that $H$ is a graph with $|H|\geq |G|-|Q|-|N|-|D|>r/8$. Since $G$ is $r$-critical and $H$ has at least one vertex, this implies that $\chi(G\setminus H)<r$. Fix an $(r-1)$-coloring $\phi\colon G\setminus H\to[r-1]$. Note that each vertex $v\in H$ is adjacent to at most $r/2+|N|+|D|<7r/8$ vertices in $G\setminus H$. Since $G$ is not $(r-1)$-colorable, this implies that $H$ is not $r/8$-list colorable. Indeed, consider the list assignment where each vertex $v\in V(H)$ is assigned the list of colors in $[r-1]$ which do not appear on its neighbors in $\phi$. These lists all have size at least $r/8$, but if there were a list coloring for these lists, it would correspond to an extension of $\phi$ to a proper $(r-1)$-coloring of $G$ which does not exist.

We are in a position to apply \cref{thm:loc-sparse-list-coloring}. We have the bound $e(H[N(v)])<6\beta r^2$ for all $v\in H$ as well as $r/8\leq \Delta(H)\leq\Delta(G)$. (The first inequality follows since $\Delta(G)\geq\delta(G)\geq r-1$ while each vertex $v\in H$ is adjacent to fewer than $7r/8$ vertices in $G\setminus H$.) Thus we have
\[\frac r8<\chi_\ell(H)\leq K\frac{\Delta(H)}{\log(\Delta(H)^2/(6\beta r^2))}\leq \frac{K\Delta(G)}{\log(1/(384\beta))}= \frac{\Delta(G)}{8C}.\]
This is a contradiction by the trivial bound $|\Delta(G)|\leq|G|\leq Cr$.
\end{proof}

\begin{prop}
\label{thm:linear-graphs}
For all $C$ and $r$ sufficiently large (in terms of $C$), the following holds. Let $G$ be an $r$-critical graph with $2r-1\leq |G|\leq Cr$. Then $t(G)>\binom r3$.
\end{prop}

\begin{proof}
Define $\beta$ to be the parameter produced in \cref{prop:large-clique} applied with $C$. Let $\delta=\beta/4$. Let $r$ be sufficiently large in terms of all of these parameters.

Suppose for the sake of contradiction that $t(G)\leq\binom r3$.

Let $G_0=G$ which satisfies $\chi(G_0)=r$. Given a graph $G_i$ satisfying $\chi(G_i)=r-i$, let $G_i'$ be an $(r-i)$-critical subgraph of $G_i$. Clearly $\chi(G_i')=r-i$. Then let $I_i$ be an independent set in $G_i'$ with maximal size. Write $k_i=|I_i|=\alpha(G_i')$ and let $G_{i+1}=G_i'\setminus I_i$. It follows that $\chi(G_{i+1})=r-i-1$. To see this, note $G_i'$ is $(r-i)$-critical and $G_{i+1}$ is a proper subgraph, so $\chi(G_{i+1})<r-i$. Furthermore, since $I_i$ is an independent set, we see that $\chi(G_i')\leq 1+\chi(G_{i+1})$ implying the desired fact. Furthermore we claim that $\chi(G_i\setminus G_j)\geq j-i$. This follows since $r-i=\chi(G_i)\leq \chi(G_j)+\chi(G_i\setminus G_j)=r-j+\chi(G_i\setminus G_j)$.

Since $G_i'$ is $(r-i)$-critical, we have the bound $\delta(G_i')\geq r-i-1$. Furthermore, we can bound the number of edges and triangles in $G$ as
\begin{equation}
\label{eq:edge-bound}
e(G)\geq\sum_{i=0}^{r-1} \sum_{v\in I_i}d_{G_i'}(v)+\frac12\sum_{i=0}^{r-1}\sum_{v\in G_i\setminus G_i'}d_{G_i}(v)
\end{equation}
and
\begin{equation}
\label{eq:tri-bound}
t(G)\geq\sum_{i=0}^{r-1} \sum_{v\in I_i}e(G_i'[N(v)])+\frac13\sum_{i=0}^{r-1}\sum_{v\in G_i\setminus G_i'}e(G_i[N(v)]).
\end{equation}

For a vertex $v\in G$, define the \emph{excess} of $v$, denoted $\ex(v)$ as follows. If $v\in I_i$ for some $i$,
\[\ex(v)=e(G_i'[N(v)])+\frac{d_{G_i'}(v)}2-\frac{(r-i-1)^2}{2k_i},\]
and if $v\in G_i\setminus G_i'$ for some $i$,
\[\ex(v)=\frac13 e(G_i[N(v)])+\frac14{d_{G_i}(v)}.\]

We combine the assumptions that $t(G)\leq\binom r3$ and $e(G)\leq\binom{Cr}2$ with \eqref{eq:edge-bound} and \eqref{eq:tri-bound} to obtain
\begin{align*}
C^2r^2
&\geq t(G)+\frac12 e(G) - \sum_{i=0}^{r-1} \frac{(r-i-1)^2}2\\
&\geq\sum_{i=0}^{r-1} \sum_{v\in I_i}\left(e(G_i'[N(v)])+\frac12d_{G_i'}(v)-\frac{(r-i-1)^2}{2k_i}\right) \\&\qquad\qquad\qquad\quad+\sum_{i=0}^{r-1}\sum_{v\in G_i\setminus G_i'}\left(\frac13e(G_i[N(v)])+\frac14d_{G_i}(v)\right)\\
&= \sum_{v\in G}\ex(v).
\end{align*}

Since $\alpha(G_i')=k_i$ and $\alpha(G_i)\leq k_{i-1}$, by Tur\'an's theorem (\cref{thm:turan}) we have the following bounds. For $v\in I_i$ for some $i$,
\begin{equation}
\label{eq:ex-indep}
\ex(v)=e(G_i'[N(v)])+\frac{d_{G_i'}(v)}2-\frac{(r-i-1)^2}{2k_i}\geq \frac{d_{G_i'}(v)^2-(r-i-1)^2}{2k_i},
\end{equation}
and for $v\in G_i\setminus G_i'$ for some $i$,
\begin{equation}
\label{eq:ex-other}
\ex(v)=\frac13 e(G_i[N(v)])+\frac14{d_{G_i}(v)}\geq  \frac{d_{G_i}(v)^2}{6k_{i-1}}.
\end{equation}
Since $d_{G_i'}(v)\geq r-i-1$ for all $i$ and $v\in I_i$, each vertex has non-negative excess.

Note that $k_1\geq k_2\geq\cdots\geq k_r>0$ satisfy $\sum k_i=|G|\leq Cr$. Pick the smallest $\ell$ such that $k_{\ell-1}\leq C\delta^{-1}$. By the previous inequality we see that $\ell\leq 1+\delta r$.

Define $H$ to be the induced subgraph of $G_\ell$ on vertices with excess at most $r^{3/2}$. Since each vertex has non-negative excess and the total excess is at most $C^2 r^2$, we see that the number of vertices in $G_{\ell} \backslash H$ is at most $C^2 r^{1/2}$. Then from $\chi(G_\ell)=r-\ell$, we conclude that $\chi(H)\geq r - \ell - C^2 r^{1/2}$. Furthermore, for each vertex $v$ in $H$ which lies in $I_i$ for some $i$, by \eqref{eq:ex-indep} we have 
\[d_{G_i'}(v)^2 \leq (r-i-1)^2 + 2r^{3/2} k_i\]
so
\[d_{G_i'}(v) \leq \sqrt{(r-i-1)^2 + 2r^{3/2} k_i}\leq \max(r - i - 1, r / 2) + 2k_i r^{1/2} \leq \max(r - i - 1, r / 2) + r^{3/4},\]
where the last inequality holds for $r$ sufficiently large in terms of $C,\delta$.

Similarly each vertex $v$ in $H$ which lies in $G_i\setminus G_i'$ for some $i$, by \eqref{eq:ex-other} we have
\[d_{G_i}(v) \leq \sqrt{6k_{i-1}r^{3/2}} \leq r^{5/6},\]
where the last inequality holds for $r$ sufficiently large in terms of $C,\delta$.

In particular, as $d_{G_i}(v) \geq d_{G_{i - 1}'}(v) - k_{i - 1} \geq r - i - k_{i - 1}$, no such $v$ lies in $G_i\setminus G_i'$ for $\ell \leq i \leq r / 2$ (again assuming that $r$ is sufficiently large).

Now take the degeneracy order on $H$ to be any total ordering which refines the partial order $\cdots\succ G_i\setminus G_i'\succ I_i\succ G_{i+1}\setminus G_{i+1}'\succ\cdots $. Set $d = r - \ell - 1 + r^{3/4}$. The results above show that $\abs{N^+(v)} \leq d$ for any $v \in H$. Furthermore, for any $i \geq \ell + C^{-1} \delta^2 d / 4$ and any $v\in I_i\cap H$ we have 
\[\abs{N^+(v)} \leq \max(r - i - 1, r / 2) + r^{3/4} \leq (1 - C^{-1}\delta^2 / 8) d\]
and for any $v\in (G_i\setminus G_i')\cap H$ we have
\[\abs{N^+(v)} \leq r^{5/6} \leq (1 - C^{-1}\delta^2 / 8) d.\]
Furthermore, for any $i$ satisfying $\ell\leq i \leq \ell + C^{-1} \delta^2 d / 4$, we have $\abs{I_i} \leq C \delta^{-1}$ and $(G_i \setminus G_{i}') \cap H = \emptyset$. Thus all but the last $\delta d / 4$ vertices in the degeneracy order satisfy $\abs{N^+(v)} \leq (1 - C^{-1}\delta^2 / 8) d$.

We are almost in a position to apply \cref{thm:degen-coloring} to $H$. We need to control $e(H[N^+(v)])$ which we can do for $v\in I_i$ with $k_i\geq 2$.

Define $b$ to be the smallest integer such that $k_b=1$. Note that this implies that $\alpha(G'_b)=1$, so $G'_b=K_{r-b}$. We have two cases. First, if $b\leq \beta r$ then $G$ is an $r$-critical graph which contains $G'_b=K_{r-b}$. By \cref{prop:large-clique} we see that $t(G)> \binom r3$, a contradiction.

Thus from now on we can assume that $b> \beta r$. Under this assumption, we claim $e(H[N^+(v)])\leq (1 - \delta) \binom d2$ for each $v\in H$. To see this, note that for $v\in I_i$ with $\ell\leq i<b$, we have $k_i\geq 2$ so by the definition of excess
\begin{align*}
e(H[N^+(v)])
&\leq e(G_i'[N(v)])\\
&\leq \ex(v)+\frac{(r-i-1)^2}{2k_i}-\frac{d_{G_i'}(v)}2\\
&\leq r^{3/2} + \frac{(r - \ell -1)^2}{4}\\
&\leq (1 - \delta) \binom d2.
\end{align*}
The last inequality holds for $r$ sufficiently large in terms of $\delta$. For $v\in I_i$ with $i\geq b$, we have $|G_i'|\leq r-b\leq(1-\beta)r$. Furthermore $d \geq (1-\delta)r$ for $r$ sufficiently large. Since we chose $\delta=\beta/4$ we see that
\[e(H[N^+(v)])\leq \binom{(1-\beta)r}2\leq (1 - \delta) \binom d2\]
holds for sufficiently large $r$.
Finally, for $v\in G_i\setminus G_i'$ by the definition of excess
\begin{align*}
e(H[N^+(v)])
&\leq e(G_i[N(v)])\\
&= 3\ex(v)-\frac{3d_{G_i}(v)}4\\
&\leq 3r^{3/2} \leq (1 - \delta) \binom d2.
\end{align*}
Thus, \cref{thm:degen-coloring} implies that $\chi(H) \leq (1 - \gamma) d$ for some constant $\gamma$ depending on $\delta, C^{-1}\delta^2 / 8$ only. Since we showed that $\chi(H) \geq r - \ell - C^2 r^{1/2}$ and we chose $d = r - \ell - 1 + r^{3/4}$, this is a contradiction for $r$ sufficiently large.
\end{proof}

\section{Proof of the main theorem}
\label{sec:main-proof}

We now prove \cref{thm:main-chromatic-weak} for graphs of all sizes. We start in the case when the graph has small clique number (less than $\epsilon r$ for some small absolute constant $\epsilon>0$). Here we get a gain over the basic Tur\'an argument by using the Ramsey--Tur\'an theorem. Then we show how to get a gain when the graph has large clique number, proving the main theorem.

\begin{lemma}
\label{lem:bounded-indep}
For all $\epsilon>0$ there exists $C,c>0$ such that the following holds. Let $G$ be an $r$-chromatic graph with $|G|\leq cr^2\log r$ vertices and $t(G)\leq \binom r3$. Then $G$ has a subgraph $G'$ with $\chi(G')\geq(1-\epsilon)r$ and $\alpha(G')\leq C$.
\end{lemma}

\begin{proof}
Take $C'= e^{1000\epsilon^{-2} K^2}$ where $K$ is the constant in \cref{thm:loc-sparse-vertex-triangle-coloring}. Take $c=\epsilon^2/(100 K^2)$. Let $C=2C'\epsilon^{-1}$.

Define the partition $G=A\sqcup B$ where $A$ is the set of vertices contained in fewer than $r^2/C'$ triangles. Clearly $|B|\leq C'r$. Write $t(A)=a r^3$ for some $a\in[0,1]$. By \cref{thm:loc-sparse-vertex-triangle-coloring} we have
\begin{align*}
r
&\leq\chi(G)\leq\chi(A)+\chi(B)\\
&\leq\chi(B)+K\left(\sqrt{\frac{|A|}{\log|A|}}+\frac{t(A)^{1/3}\lg\lg(t(A)^2/(r^2/C')^3)}{\lg^{2/3}(t(A)^2/(r^2/C')^3)}\right)\\
&\leq\chi(B)+K\left(\sqrt{\frac{cr^2\log r}{\log(cr^2\log r)}}+\frac{a^{1/3}r\lg\lg(a^2C'^3)}{\lg^{2/3}(a^2C'^3)}\right)\\
&\leq \chi(B)+\epsilon r/2.
\end{align*}
To see the last inequality, first note that our choice of $c$ makes the first term smaller than $\epsilon r/4$. The second term is upper-bounded by $Ka^{1/3}r$, which is sufficiently small if $a<(\epsilon/4K)^3$. In the range $(\epsilon/4K)^3\leq a\leq 1$, the function agrees with the non-truncated version which is increasing, so in this range
\[K\frac{a^{1/3}r\lg\lg(a^2C'^3)}{\lg^{2/3}(a^2C'^3)}= K\frac{a^{1/3}r\log\log(a^2C'^3)}{\log^{2/3}(a^2C'^3)}\leq K\frac{r\log\log(C'^3)}{\log^{2/3}(C'^3)}\leq \epsilon r/4.\] The last inequality holds for our choice of $C'$.

Since $\chi(B)\leq r$, we can construct a partition $B=I_1\sqcup\cdots\sqcup I_r$ into independent sets as follows. Define $I_j$ to be an independent set in $G\setminus(I_1\sqcup\cdots\sqcup I_{j-1})$ with maximal size. Say that $|I_j|=k_j$ so $k_1\geq k_2\geq\cdots\geq k_r\geq 0$ and $\sum k_i=|B|\leq C'r$. Therefore we see that $k_{\epsilon r/2}\leq 2C'\epsilon^{-1}$. Let $G'$ be the union of $I_j$ for $j\geq\epsilon r/2+1$. Then by construction $\alpha(G')=k_j\leq 2C'\epsilon^{-1}=C$. Furthermore, by construction $\chi(G)\leq \chi(G')+\epsilon r$, so $\chi(G')\geq(1-\epsilon)r$.
\end{proof}

We use the following Ramsey--Tur\'an type result.

\begin{thm}[{Erd\H{o}s--S\'os \cite{ES70}}]
\label{prop:ramsey-turan}
For all $\delta>0$ and even $k$ there exists $\epsilon>0$ such that the following holds. Suppose $\alpha(G)\leq k$ and $\omega(G)\leq \epsilon \abs{G}$. Then $e(G)\geq \left(\tfrac1k-\delta\right)|G|^2-\tfrac12|G|$.
\end{thm}

\begin{prop}
\label{thm:small-clique}
There exists $\epsilon>0$ such that the following holds. For sufficiently large $r$, let $G$ be an $r$-chromatic graph with $e(G)\leq r^3/1000$. If $\omega(G)\leq \epsilon r$ then $t(G)\geq\binom r3$. 
\end{prop}

\begin{proof}
Let $C,c$ be the constants produced by applying \cref{lem:bounded-indep} with parameter $1/10$. Define $10\epsilon$ to be the minimum of the constants produced by applying \cref{prop:ramsey-turan} with parameters $1/20C, k$ for all even $k$ less than $C+1$. 

Let $G$ be an $r$-critical graph with $e(G)\leq r^3/1000$ and $\omega(G)\leq\epsilon r$. Since $G$ is $r$-critical it must have minimum degree at least $r-1$, so $|G|\leq 2e(G) / (r - 1) \leq r^2/100$. Assume $r$ is sufficiently large so that this is smaller than $cr^2\log r$. Suppose for contradiction that $t(G)<\binom r3$. Then by \cref{lem:bounded-indep}, there is a subgraph $G_0$ of $G$ with $\chi(G_0)=r'\geq9r/10$ and $\alpha(G')\leq C$.

Given a graph $G_i$ satisfying $\chi(G_i)=r'-i$, let $G_i'$ be an $(r'-i)$-critical subgraph of $G_i$. Then let $I_i$ be an independent set in $G_i'$ with maximal size. Let $G_{i+1}=G_i'\setminus I_i$. It follows that $\chi(G_{i+1})=r'-i-1$.

Write $k_i=\alpha(G_i')\leq C$. We know $\delta(G_i')\geq r'-i-1$. For $i<4r/5$, we know that $\delta(G_i')\geq r/10$. Let $k_i'=2\lceil\tfrac{k_i}2\rceil\leq k_i+1$. We know that $k_i\geq 2$ since if $k_i=1$, then $G_i'$ is a clique of size $r'-i\geq r/10$ which contradicts our assumption that $\omega(G)\leq \epsilon r$. Thus $k_i'\leq 4k_i/3$. Now by \cref{prop:ramsey-turan}, for each $v\in G_i'$ we have
\[e(G_i'[N(v)])\geq \left(\frac 1{k_i'}-\frac1{20C}\right)d_{G_i'}(v)^2-\frac{d_{G_i'}(v)}2\geq\frac{7(r'-i-1)^2}{10k_i}-\frac{d_{G_i'}(v)}2.\]
We can apply \cref{prop:ramsey-turan} to $G_i'[N(v)]$ since $\omega(G_i'[N(v)])\leq \epsilon r\leq 10\epsilon(r'-i-1)$ for $i<4r/5$.

For $i\geq 4r/5$ we have the simple Tur\'an bound that
\[
e(G_i'[N(v)])\geq \frac{d_{G_i'}(v)^2}{2k_i}-\frac{d_{G_i'}(v)}2\geq\frac{(r'-i-1)^2}{2k_i}-\frac{d_{G_i'}(v)}2.
\]

Summing this bound over all $v\in I_i$ and all $1\leq i\leq r'$ gives
\[t(G)+\frac12e(G)\geq \sum_{i=0}^{4r/5-1}\frac7{10}(r'-i-1)^2+\sum_{i=4r/5}^{9r/10}\frac12(r'-i-1)^2>1.02\binom r3.\]

Since we assumed that $e(G)\leq r^3/1000$, this implies that $t(G)\geq 1.01\binom r3$.
\end{proof}

\begin{proof}[Proof of \cref{thm:main-chromatic-weak}]
Take $\epsilon$ to be the constant in \cref{thm:small-clique}. Define $C=1000\epsilon^{-6}$. Take $r_0$ to be such that \cref{thm:linear-graphs} holds for this choice of $C$ and $\epsilon^{2/3}r/4\geq r_0$. Define $K=\max\{4r_0\epsilon^{-2/3},32000\epsilon^{-2}\}$. We can assume that $r\geq K\geq 4r_0\epsilon^{-2/3}$ since the theorem is trivially true when $r<K$.

Let $G_0$ be a subgraph of $G$ which satisfies $\chi(G_0)=r$. Given a graph $G_i$ satisfying $\chi(G_i)=r-i$, let $G_i'$ be an $(r-i)$-critical subgraph of $G_i$. For a set $U\subseteq V(G_i')$, define the \emph{modified weight of $U$} to be $\abs{U}$ if $U$ is disjoint from every clique of size $\epsilon (r-i)$ in $G_i'$ and $(1 + \epsilon^2/10) \abs{U}$ otherwise. Then let $I_i$ be an independent set in $G_i'$ with maximal \emph{modified weight}. Let $G_{i+1}=G_i'\setminus I_i$. It follows that $\chi(G_{i+1})=r-i-1$. We stop the process if any of the following occur:
\begin{enumerate}[(i)]
    \item $i>(1-\epsilon^{2/3}/4)r$; or
    \item $\abs{G_i'} < C(r - i)$; or
    \item $\omega(G_i')<\epsilon (r-i)$.
\end{enumerate}

Suppose we have not stopped the process by step $i$. First note that this implies that $|G_i'|\geq C(r-i)$. Thus $\alpha(G_i')\geq \abs{G_i'} / \chi(G_i')\geq C$.

\textbf{Case 1:} $I_i$ is disjoint from every clique of size $\epsilon (r-i)$ in $G_i'$. Since we have not stopped the process yet, $\omega(G_i')\geq \epsilon (r-i)$. Let $Q$ be a clique in $G_i'$ with maximal size. Each  $u \in Q$ is adjacent to at least $\tfrac{\epsilon^2}{20}\abs{I_i}$ elements of $I_i$, else adding $u$ to $I_i$ and removing all its neighbors would create an independent set with greater modified weight. Therefore, the number of edges $e(I_i, Q)$ between $I_i$ and $Q$ is at least $\tfrac{\epsilon^2}{20}\abs{I_i} \abs{Q}$. Note that $\abs{I_i}=\alpha(G_i') \geq C$.
Thus we conclude that
\[t_{G'_i}(I_i) \geq \sum_{v \in I_i}\binom{|N(v)\cap Q|}{2} \geq \abs{I_i} \binom{e(I_i, Q) / \abs{I_i}}{2} \geq C \binom{\epsilon^3 (r-i) / 20}{2} \geq 2 \binom{r-i-1}{2},\] where the last inequality follows by our choice of $C,\epsilon$.

\textbf{Case 2:} There exists a vertex $w\in I_i$ that is in a clique $Q$ of size at least $\epsilon(r-i)$ in $G_i'$. In this case, for any $v \in I_i \backslash \{w\}$, by Tur\'{a}n's theorem the number of triangles in $G'_i$ containing $v$ is at least $(r - i-1)^2 / (2\alpha(G_i')) - d_{G_i'}(v).$
Furthermore, $w$ is contained in at least $\binom{\abs{Q} - 1}{2} \geq \binom{\epsilon (r-i) - 1}{2} \geq \tfrac{\epsilon^2 }2  \binom{r-i}{2}$ triangles. Thus
\[t_{G_i'}(I_i) \geq \sum_{v \in I_i \backslash \{w\}} \left( \frac{(r - i-1)^2}{2\alpha(G_i')} - d_{G_i'}(v)\right) + \frac{\epsilon^2}{2} \binom{r-i}{2}.\]
By the maximality of the modified weight of $I_i$, we have $\alpha(G_i') \leq (1 + \epsilon^2/10) \abs{I_i}$. We conclude that
\[t_{G_i'}(I_i) \geq (|I_i|-1)\left(\frac{(r-i-1)^2}{2(1+\epsilon^2/10)|I_i|}\right)+\frac{\epsilon^2}2\binom{r-i}2-m_i\geq \left(1 + \frac{\epsilon^2}{4}\right)\binom{r - i}{2} - m_i,\]
where $m_i$ is the number of edges incident to $I_i$ in $G_{i}'$. The second inequality follows from our choice of $C, \epsilon$ since $|I_i|\geq \alpha(G_i')/(1+\epsilon^2/10)\geq C/(1+\epsilon^2/10)$.

Note that the above inequality holds in either case 1 or case 2.

Let $G_j'$ be the graph we stop at, with chromatic number $r - j$. Telescoping the change in the number of triangles, we have
\[t(G) - t(G_j') \geq\sum_{i = 0}^{j - 1} \left(1 + \frac{\epsilon^2}{4}\right)\binom{r - i}{2} - m',\]
where $m' = \abs{E(G)} - \abs{E(G_j')}$. As each edge is in at least $K$ triangles, we have
\[t(G) - t(G_j') \geq \frac{K}{3} m'.\]
Taking a weighted average of the two inequalities (a factor of $K/(3+K)$ times the first inequality and $3/(3+K)$ times the second), we also get
\[t(G) - t(G_j') \geq \left(1+\frac{\epsilon^2}{8}\right)\sum_{i = 0}^{j - 1}\binom{r - i}{2}.\]

Suppose we stop because of (i). Then $j\geq (1-\epsilon^{2/3}/4)r$, so
\[t(G)\geq \left(1+\frac{\epsilon^2}{8}\right) \sum_{i = 0}^{j - 1}\binom{r - i-1}{2}=\left(1+\frac{\epsilon^2}{8}\right)\left(\binom r3-\binom {r-j}3\right)\geq\binom r3.\]

If we stop because of (ii), then we know that $\abs{G_j'} < C(r-j)$. We apply \cref{thm:linear-graphs} or \cref{cor:gallai-triangle} to $G_j'$, which is $(r-j)$-critical, depending on whether $\abs{G_j'} \leq 2(r - j) - 2$ or not. (Note that we can apply \cref{thm:linear-graphs} since $(r-j)\geq \epsilon^{2/3} r/4\geq r_0$.) This implies that $t(G_j')\geq\binom{r-j}3$. Thus
\begin{align*}
t(G)
&\geq\left(1+\frac{\epsilon^2}{8}\right)\sum_{i=0}^{j-1}\binom{r - i-1}{2} +\binom{r-j}3 \geq \binom r3.
\end{align*}

If we stop because of (iii), then $\omega(G_j')<\epsilon (r-j)$. Applying \cref{thm:small-clique} to $G_j'$, we conclude that either  $t(G_j')\geq\binom{r-j}3$ or $e(G_j')\geq (r-j)^3/1000$. In the former case, the same computation as above implies that $t(G)\geq\binom r3$. In the latter case, since every edge of $G_j'$ is contained in at least $K$ triangles in $G$, we conclude that $t(G)\geq K e(G_j')/3\geq K(r-j)^3/3000$. This suffices to prove $t(G)\geq\binom r3$ since $j<(1-\epsilon^{2/3}/4)r$ and $K\geq 32000 \epsilon^{-2}$.
\end{proof}

\section{Concluding remarks}
\label{sec:conclusion}

A natural next step would be to show that \cref{thm:main-triangle-ramsey} holds for all $t$, not just those that are sufficiently large. Though it would not immediately imply new Ramsey results, we also conjecture that the following variant of \cref{thm:main-chromatic-weak} holds.

\begin{conj}\label{conj:manytriangles}
There exists $c>0$ such that every $r$-chromatic graph with at most $cr^3\log^2 r$ edges contains at least $\binom r3$ triangles.
\end{conj}

\cref{conj:manytriangles} would imply that the number of triangles guaranteed in an $r$-critical graph with a given number of edges has a transition from at least $\binom{r}{3}$ to $0$ when the number of edges is $\Theta(r^3\log^2 r)$. 

More generally, we are also interested in studying \cref{conj:clique-clique}, showing that $r_{K_s}(K_t)=\binom{r(K_t)}{s}$ for all $s\leq t$. The condition $s\leq t$ is necessary, since Folkman \cite{Fol70} constructed an explicit family of graphs showing $r_{K_s}(K_t)=0$ for all $s>t$. A more modern perspective is given by R\"odl and Ruci\'nski \cite{RR95} who showed that $G(n,p)$ is $K_t$-Ramsey with probability $1-o(1)$ when $p=Cn^{-2/(t+1)}$. This graph is easily seen to have few copies of $K_{t+1}$. Modifying the graph appropriately, one can construct $K_{t+1}$-free graphs that are $K_t$-Ramsey.

R\"odl and Ruci\'nski constructed $F$-free graphs that are $H$-Ramsey for many pairs $(F,H)$. For a graph $G$, define its 2-density as
\[m_2(G) = \max_{G' \subseteq G} \frac{e(G') - 1}{|G'| - 2},\]
where $G'$ ranges over all induced subgraphs of $G$ with at least 3 vertices.

\begin{prop}[{\cite[Corollary 2]{RR95}}]
If $m_2(F)>m_2(H)$, then $r_F(H)=0$.
\end{prop}

It is an interesting question to determine for which pairs $(F,H)$ is it true that $r_F(H)=0$. We remark that a very recent result of Reiher and R\"{o}dl \cite{RR23} is equivalent to the statement that $r_F(H) = 0$ for all pairs of graphs $(F,H)$ where the girth of $F$ is smaller than the girth of $H$.

One class of $(F,H)$ for which determining $r_F(H)$ is particularly easy is that when $F$ is a forest and $H$ is a clique. This generalizes the classical determination of the size Ramsey number of cliques (\cref{thm:chvatal-size-ramsey-clique}).

\begin{prop}
For every forest $F$ and positive integer $t$, the complete graph on $r(K_t)$ vertices has the minimum number of copies of $F$ over all graphs which are Ramsey for $K_t$. 
\end{prop}

By \cref{prop:ramsey-critical}(a), any graph that is $K_t$-Ramsey is $r(K_t)$-chromatic and thus contains a subgraph with minimum degree at least $r(K_t)-1$. Thus the above proposition follows from the following easy lemma.

\begin{lemma}
Let $F$ be a forest. If a graph $G$ has minimum degree at least $d$, then the number of (unlabeled) copies of $F$ in $G$ is at least the number of (unlabeled) copies of $F$ in $K_{d + 1}$.
\end{lemma}

\begin{proof}
Label the vertices of $F$ as $\{1,2,\ldots,|F|\}$ so that vertex $i$ has at most one neighbor among $\{1,2,\ldots,i - 1\}$ for each $i$. We construct a labeled copy of $F$ in $G$ one vertex at a time. Since $|G|\geq d+1$, there are at least $d-i+2$ ways to pick vertex $i$. Thus the number of labeled copies of $F$ in $G$ is at least $\prod_{i=1}^{|F|}(d-i+2)$ which is the number of labeled copies of $F$ in $K_{d+1}$. Dividing by $|\mathrm{Aut}(F)|$ gives that the number of unlabeled copies of $F$ in $G$ is at least the number of unlabeled copies of $F$ in $K_{d+1}$.
\end{proof}

Going back to \cref{conj:clique-clique}, the upper bound $r_{K_s}(K_t)\leq \binom{r(K_t)}s\leq 4^{st}$ follows immediately from the definition and can be improved to $r_{K_s}(K_t) \leq 3.9995^{st}$ via recent work on diagonal Ramsey numbers \cite{CGMS23}. While we are currently unable to prove a matching lower bound, we give the following result which complements the classical bound $r(K_t)>2^{t/2}$.

Let $m(b)$ be the minimum number of edges in a $b$-uniform hypergraph without property B, i.e., the minimum number of edges in a $b$-uniform hypergraph without a proper two-coloring of its vertices.

\begin{prop}
\label{prop:clique-clique-lower}
For all $s\leq t$, letting $b=\binom{t}{2}$, we have 
\[r_{K_s}(K_t)\geq m(b)^{s/t}\geq 2^{s(t-1)/2-1}.\]
\end{prop}

\begin{proof}
For the case when $s=t$, let $G$ be a $K_t$-Ramsey graph. Define $H$ to be the $b$-uniform hypergraph whose vertex set is $E(G)$ and whose edges correspond to copies of $K_t$ in $G$. Then this hypergraph does not have property B, since any proper 2-coloring of the vertices of $H$ corresponds to a 2-edge-coloring of $G$ without a monochromatic $K_t$. Therefore we conclude that the number of copies of $K_t$ in $G$, which is $|E(H)|$, is at least $m(b)$.

Now for $s\leq t$, we know that every graph that is $K_t$-Ramsey contains at least $m(b)$ copies of $K_t$. By the Kruskal--Katona theorem, such a graph also contains at least $m(b)^{s/t}$ copies of $K_s$.
\end{proof}

We do not believe that the techniques in this paper will be able to prove \cref{conj:clique-clique}, as there is no apparent analogue of \cref{thm:main-chromatic-weak} or even \cref{thm:basic-turan-bound} for $K_s$ with $s\geq 4$. To see this, consider any triangle-free $r$-chromatic graph $G$ with $O(r^3\log^{2}r)$ edges. Modifying this graph by placing a disjoint copy of $K_{s+K}$ on each edge for an appropriately chosen $K$ gives counterexamples to many natural analogues of \cref{thm:main-chromatic-weak,thm:basic-turan-bound}.

%\bibliographystyle{amsplain0.bst}
%\bibliography{bib}

\end{document}